\documentclass[12pt,twoside]{amsart}
\usepackage{amssymb,amscd}
\usepackage{mathrsfs}
\usepackage{array,float}

\usepackage{mathtools}
\DeclarePairedDelimiter\floor{\lfloor}{\rfloor}

\theoremstyle{plain}
\newtheorem{thm}{Theorem}[section]

\newtheorem{pro}[thm]{Proposition}
\newtheorem{cor}[thm]{Corollary}

\theoremstyle{remark}
\newtheorem{rem}[thm]{Remark}
\newtheorem*{rem*}{Remark}

\theoremstyle{definition}

\numberwithin{equation}{section}

\newcommand{\N}{\mathbb{N}}

\newcommand{\ackn}{  \noindent{\sc Acknowledgement }\hspace{5pt} }

\renewcommand{\phi}{\varphi}

\begin{document}

\title[Uncountably many non-commensurable finitely presented pro-$p$ groups]{Uncountably many non-commensurable finitely presented pro-$p$ groups}

\thanks{This research was partially supported by CNPq}

\author{Ilir Snopce}
\address{Universidade Federal do Rio de Janeiro\\
  Instituto de Matem\'atica \\
  21941-909 Rio de Janeiro, RJ \\ Brasil }
\email{ilir@im.ufrj.br}

\begin{abstract}
 Let $m\geq 3$ be a positive integer.  We prove that there are uncountably many non-commensurable metabelian uniform pro-$p$ groups of dimension $m$. Consequently, there are uncountably many non-commensurable finitely presented pro-$p$ groups with minimal number of generators  $m$ (and minimal number of relations  $ {m \choose 2}$). 
\end{abstract}


\maketitle

\section{Introduction}
Throughout let $p$ be a prime.  In~\cite{La65}, Lazard gave a comprehensive treatment of the theory of $p$-adic analytic groups. Later Lubotzky and Mann re-interpreted the group-theoretic aspects of Lazard's work by introducing the concept of  powerful pro-$p$ group (see \cite{LuMa87}). A pro-$p$ group $G$ is said to be \emph{powerful} if $p \geq 3$ and
$[G,G]\leq G^p$, or $p=2$ and $[G,G]\leq G^4$.  Here, $[G,G]$ and
$G^p$ denote the (closures of the) commutator subgroup and the
subgroup generated by all $p$th powers. Using this terminology, Lazard's main result is the following algebraic characterisation of $p$-adic analytic groups: a topological group is $p$-adic analytic if
and only if it contains an open subgroup which is a finitely generated powerful pro-$p$ group.

\smallskip

 A powerful pro-$p$ group $G$ is called \emph{uniform} if it is finitely generated and torsion-free. A key feature of a uniform pro-$p$ group $G$ is that its minimal number of (topological) generators $d(G)$  coincides with the dimension of $G$ as a $p$-adic manifold. 

\smallskip
 
 In  this short note we prove the following.

\begin{thm}
\label{MainTheorem}
Let $m\geq 3$ be a positive integer.  There are uncountably many non-commensurable metabelian uniform pro-$p$ groups of dimension $m$. Consequently, there are uncountably many non-commensurable finitely presented pro-$p$ groups with minimal number of generators  $m$ (and minimal number of relations $
  {m \choose 2}$).
\end{thm}

As a direct consequence we get the following.

\begin{cor}\label{MainCorollary}
Let $m\geq 3$ be a positive integer. There are uncountably many non-commensurable finitely presented pro-$p$ groups $G$ with $d(G)=m$ which are not a pro-$p$ completion of any finitely presented abstract group. In particular, these groups do not have a finite presentation with all the relations of finite length (as words on the generators).
\end{cor}

Let $\Gamma$ be a finitely generated abstract group or a finitely generated profinite group, and for $n\in \mathbb{N}$, let $\mathit{a}_n(\Gamma)$ denote the number of subgroups of index $n$ in $\Gamma$. The zeta function of  $\Gamma$ is given by the Dirichlet series
\begin{displaymath}
\zeta_{\Gamma}(s)=\sum_{n=1}^{\infty}\mathit{a}_n(\Gamma) n^{-s}.
\end{displaymath}
Similarly, the normal zeta function of $\Gamma$  is defined by
\begin{displaymath}
\zeta_{\Gamma}^\triangleleft(s)=\sum_{n=1}^{\infty}\mathit{a}_{n}^\triangleleft(\Gamma) n^{-s},
\end{displaymath}
where $\mathit{a}_{n}^\triangleleft(\Gamma)$ denotes the number of normal subgroups of index $n$ in $\Gamma$.

We say that two groups $H$ and $K$ are \emph{isospectral} (\emph{normally isospectral}) if their zeta functions (normal zeta functions) are the same.  In \cite{Snopceil}  is given an example of an infinite family of non-commensurable normally isospectral pro-$p$ groups.

 Du Sautoy has proved that if $G$ is a compact $p$-adic analytic group then  $\zeta_G(s)$ and $\zeta_G^\triangleleft(s)$ are  rational functions over $\mathbb{Q}$ of $p^{-s}$ (see \cite{Sautoypaper}). This result together with Theorem \ref{MainTheorem} imply the following.
 
 \begin{cor}\label{Zeta}
There are uncountably many non-commensurable isospectral (normally isospectral) pro-$p$ groups.
\end{cor}

The question of whether there are uncountably many non-isomorphic finitely presented pro-$p$ groups was recently explicitly raised by A. Lubotzky at the conference ``Geometric and Combinatorial Group Theory" in honor of E. Rips and by E. Zelmanov at the conference ``XX Coloquio Latinoamericano de \'{A}lgebra" (Zelmanov attributed the question to Lubotzky) . I am grateful to E. Zelmanov and R. Grigorchuk for encouraging me to write down this paper.

\section{Powerful pro-$p$ groups and Lie algebras}
Let $G$ be a pro-$p$ group.  The lower central
$p$-series of $G$ is defined as follows: $P_1(G)=G$ and $P_{i+1}(G) =
{P_i(G)}^p [P_i(G), G]$ for $i \in \N$. 
One can use this series to define uniform pro-$p$ groups. Indeed, a pro-$p$ group is uniform if and only if it is finitely generated, powerful and $\lvert P_i(G) : P_{i+1}(G) \rvert = \lvert G : P_2(G)
    \rvert$ for all $i \in \N$. This definition of a uniform pro-$p$ group is equivalent to the definition given in the introduction (see \cite[Theorem 4.5]{DidSMaSe99}).
    
     Given a powerful pro-$p$ group $G$ and $n\in \mathbb{N}$, we have $P_{n+1}(G)=G^{p^n}=\{ x^{p^n} ~|~ x\in G \}$ (see \cite[Theorem 3.6]{DidSMaSe99}). Moreover, if $G$ is uniform, then the mapping $ x \mapsto x^{p^n}$ is a homeomorphism from $G$ onto $G^{p^n}$ (see \cite[Lemma 4.10]{DidSMaSe99}). This shows that each element $x\in G^{p^n}$ admits a unique $p^n$th root in $G$, which we denote by $x^{p^{-n}}$. 
    
    Analogous to the case of pro-$p$ groups, a $\mathbb{Z}_p$-Lie algebra $L$ is called \emph{powerful} if $L\cong \mathbb{Z}_p^d$ for some $d>0$ as $\mathbb{Z}_p$-module and $(L,L)_{Lie}\subseteq pL$ ( $(L,L)_{Lie} \subseteq 4L$ if $p=2$).

 If $G$ is an analytic pro-$p$ group, then it has a characteristic open subgroup which is uniform. For every open uniform subgroup $H \leq G$, $\mathbb{Q}_p[H]$ can be made into a normed $\mathbb{Q}_p$-algebra, call it $\mathit{A}$, and $log(H)$, considered as a subset of the completion $\hat{A}$ of $A$, will have the structure of a Lie algebra over $\mathbb{Z}_p$. 
There is a different construction of an intrinsic Lie algebra over $\mathbb{Z}_p$ for uniform groups. The uniform group $U$ and its Lie algebra over $\mathbb{Z}_p$, call it $L_U$, are identified as sets, and the Lie operations are defined by
\begin{displaymath}
g+h=\lim_{n \to \infty}(g^{p^n}h^{p^n})^{p^{-n}},  ~ ~ ~  (g,h)_{Lie}=\lim_{n \to \infty}[g^{p^n},h^{p^n}]^{p^{-2n}}=\lim_{n\to \infty} (g^{-p^n}h^{-p^n}g^{p^n}h^{p^n})^{p^{-2n}} .
\end{displaymath} It turns out that $L_U$ is a powerful $\mathbb{Z}_p$-Lie algebra and it is isomorphic to the $\mathbb{Z}_p$-Lie algebra $log(U)$.
On the other hand, if $L$ is a powerful $\mathbb{Z}_p$-Lie algebra, then the Campbell-Hausdorff formula induces a group structure on $L$; the resulting group is a uniform pro-$p$ group. If this construction is applied to the $\mathbb{Z}_p$-Lie algebra $L_U$ associated to a uniform group $U$, one recovers the original group. Indeed, the assignment $U\mapsto L_U$ gives an isomorphism between the category of uniform pro-$p$ groups and the category of powerful $\mathbb{Z}_p$-Lie algebras (see \cite[Theorem 9.10]{DidSMaSe99}).  A
detailed treatment of $p$-adic analytic groups is given in~\cite{DidSMaSe99}.

\section{Proof of Theorem 1.1}

Let us denote by $\mathbb{Z}_p^{*}$ the group of units of the $p$-adic integers, i.e., $\mathbb{Z}_p^{*} = \mathbb{Z}_p \setminus p\mathbb{Z}_p$.  

\begin{pro}\label{LieAlgebras}
Let $d\in \mathbb{Z}_p^{*}$ and for $n\geq 1$, let $L_{2n+1}(d)$ and $L_{2n+2}(d)$ be $\mathbb{Z}_p$-Lie algebras  defined in the following way.
\begin{itemize}
\item [(i)] $L_{2n+1}(d)$ is the free  $\mathbb{Z}_p$-module on the basis $\{x, e_2, ..., e_{2n+1}\}$ and the  Lie bracket is given as follows:
 \begin{displaymath}
  [e_i,e_j]=0 \textrm{ for } 2\leq i, j \leq 2n+1, ~ [e_2, x]=de_{2n+1}, 
 \end{displaymath}
 \begin{displaymath}
  [e_i, x]=e_{2n-i+3} \textrm{ for } 3\leq i \leq 2n \textrm{ and } [e_{2n+1}, x] = e_2 +e_{2n+1}. 
 \end{displaymath}
\item [(ii)] $L_{2n+2}(d)$ is the free  $\mathbb{Z}_p$-module on the basis $\{x, e_2, ..., e_{2n+2}\}$ and the Lie  bracket is given as follows: 
 \begin{displaymath}
  [e_i,e_j]=0 \textrm{ for } 2\leq i, j \leq 2n+2, ~ [e_2, x]=de_{2n+2},
 \end{displaymath}
 \begin{displaymath}
   \textrm{ and } [e_i, x]=e_{2n-i+4} \textrm{ for } 3\leq i \leq 2n+2. 
 \end{displaymath}
\end{itemize}
Then $L_k(d) \cong L_k(l)$ if and only if $d=l$. Moreover, $d$ is an invariant of the isomorphism type of the $\mathbb{Q}_p$-Lie algebra $L_k(d)\otimes_{\mathbb{Z}_p}\mathbb{Q}_p$.
\end{pro}
\begin{proof}
Let $k \geq 3$ and let $L=L_k(d)$. It is easy to see that $L$ is a well-defined $\mathbb{Z}_p$-Lie algebra and  that $L'=[L,L]$ is an abelian ideal generated by $e_2, e_3,..., e_k$. Let $A(x)$ denote the restriction of $\textrm{ad}x$ to $L'$ and let $A_k(d)$ be the matrix associated to this linear transformation with respect to the basis $e_2,...,e_k$. If $k=2n+1$ we have the following ${2n}\times {2n}$ matrix
\begin{displaymath}
A_{2n+1}(d)=
\begin{pmatrix} 0 & \cdots & 0 & 0 & d\\ 0 & \cdots & 0 & 1 & 0\\ \vdots & \ddots & \ddots & \ddots & \vdots \\ 0 & 1 & 0 & \cdots & 0\\ 1 & 0 & \cdots & 0 & 1 \end{pmatrix}  
\end{displaymath}
and if $k=2n+2$ we have the following $(2n+1) \times (2n+1)$ matrix
\begin{displaymath}
A_{2n+2}(d)=
\begin{pmatrix} 0 & \cdots & 0 & 0 & d\\ 0 & \cdots & 0 & 1 & 0\\ \vdots & \ddots & \ddots & \ddots & \vdots \\ 0 & 1 & 0 & \cdots & 0\\ 1 & 0 & \cdots & 0 & 0 \end{pmatrix}. 
\end{displaymath}
Note that $\textrm{tr}(A_k(d))=1$ for all $k\geq 3$. Moreover, $\textrm{det}(A_{2n+1}(d))={(-1)}^nd$ and $\textrm{det}(A_{2n+2}(d))={(-1)}^nd$ for all $n\geq 1$ (i.e., $\textrm{det}(A_k(d))={(-1)}^{\floor{\frac{k-1}{2}}}d$ for $k\geq 3$).\\
 Also note that $L=L'\oplus x\mathbb{Z}_p$. Moreover, if also  $L=L'\oplus y\mathbb{Z}_p$ for some $y\in L$, then $y=ux + e$ with $u \in \mathbb{Z}_p^{*}$ and $e\in  L'$, and
\begin{displaymath}
[e_i,y]=[e_i,ux+e]=u[e_i,x]+[e_i,e]=u[e_i,x].
\end{displaymath} 
Thus $A(y) =\textrm{ad}y_{|L'} = uA(x)$, so $\textrm{tr}A(y) = u\textrm{tr}A(x)$ and $\textrm{det} A(y) = u^{k-1} \textrm{det}A(x)$, where $k-1 = \textrm{dim}L'$. Since $\textrm{tr}A(x) = \textrm{tr}(A_k(d))=1$, it follows that $${(\textrm{tr}A(y))}^{-(k-1)}\cdot \textrm{det} A(y) = {u}^{1-k}\cdot \textrm{det} A(y) = \textrm{det}A(x) = {(-1)}^{\floor{\frac{k-1}{2}}}d$$ is an invariant of $L = L_k(d)$. Thus $L_k(d) \cong L_k(l)$ if and only if $d=l$.

Finally, note that our proof works equally well with $\mathbb{Q}_p$ in place of $\mathbb{Z}_p$. Thus  $d$ is an invariant of the isomorphism type of the $\mathbb{Q}_p$-Lie algebra $L_k(d)\otimes_{\mathbb{Z}_p}\mathbb{Q}_p$.
\end{proof}

\begin{cor}\label{PowerfulAlgebra}
Let $k\geq 3$ and let $d, l \in \mathbb{Z}_p^{*}$. The $\mathbb{Z}_p$-Lie algebra $p^2L_k(d)$ is powerful and $p^2L_k(d)\otimes_{\mathbb{Z}_p}\mathbb{Q}_p\cong p^2L_k(l)\otimes_{\mathbb{Z}_p}\mathbb{Q}_p$ if and only if $d=l$. In particular, there are uncountably many  non-isomorphic (powerful) $\mathbb{Z}_p$-Lie algebras of rank $k$. 
\end{cor}
\begin{proof}
We have $[p^2L_k(d),p^2L_k(d)]=p^4[L_k(d),L_k(d)]\subseteq p^2(p^2L_k(d))$. Hence, $p^2L_k(d)$ is a powerful $\mathbb{Z}_p$-Lie algebra. Now note that $f_1, ..., f_k$ is a basis of $L_k(d)\otimes_{\mathbb{Z}_p}\mathbb{Q}_p$ if and only if $p^2f_1,...,p^2f_k$ is a basis of $p^2L_k(d)\otimes_{\mathbb{Z}_p}\mathbb{Q}_p$. Hence, $p^2L_k(d)\otimes_{\mathbb{Z}_p}\mathbb{Q}_p \cong p^2L_k(l)\otimes_{\mathbb{Z}_p}\mathbb{Q}_p$ if and only if $L_k(d)\otimes_{\mathbb{Z}_p}\mathbb{Q}_p \cong L_k(l)\otimes_{\mathbb{Z}_p}\mathbb{Q}_p$, and by Proposition \ref{LieAlgebras}, if and only if $d=l$. The last part of the corollary follows directly from the fact that $\mathbb{Z}_p^{*} = \mathbb{Z}_p \setminus p\mathbb{Z}_p$ is an uncountable set.  
\end{proof}
We need the following well-known proposition.
\begin{pro}[{\cite[Proposition 4.32]{DidSMaSe99}}]\label{FinitelyPresented}
Let $G$ be a uniform pro-$p$ group of dimension $d=d(G)$ and let $X=\{x_1,..., x_d  \}$ be a topological generating set for $G$. Then $G$ has a presentation $\langle X; R \rangle$, where 
\begin{displaymath}
R= \{[x_i,x_j]x_1^{a_1(i,j)}\cdots x_d^{a_d(i,j)} ~ | ~ 1\leq i,j \leq d   \},
\end{displaymath} 
and for each $m, i$ and $j$, $a_m(i,j) \in p\mathbb{Z}_p$ if $p$ is odd, $a_m(i,j) \in 4\mathbb{Z}_2$ if $p=2$. In particular, $G$ is finitely presented.
\end{pro}
\begin{proof}[Proof of Theorem \ref{MainTheorem}]
Let $m\geq 3$ be a positive integer. By Corollary \ref{PowerfulAlgebra}, we know that $p^2L_m(d)$ is a powerful $\mathbb{Z}_p$-Lie algebra of rank $m$ for all $d \in \mathbb{Z}_p^{*}$. By {\cite[Theorem 9.8]{DidSMaSe99}}, we can associate to $p^2L_m(d)$ a uniform pro-$p$ group $G_m(d)$ which has the same underlying set as $p^2L_m(d)$ and such that $d(G_m(d))=m$. By {\cite[Theorem 9.10]{DidSMaSe99}} it follows that $G_m(d) \cong G_m(l)$ if and only if $p^2L_m(d) \cong p^2L_m(l)$. Moreover, $G_m(d)$ and $G_m(l)$ are pairwise non-commensurable whenever $d \neq l$, since $p^2L_k(d)\otimes_{\mathbb{Z}_p}\mathbb{Q}_p \cong p^2L_k(l)\otimes_{\mathbb{Z}_p}\mathbb{Q}_p$ if and only if $d=l$. 

 Now by Corollary \ref{PowerfulAlgebra}, we have uncountably many uniform pro-$p$ groups $G$ such that $d(G)=m$ and, by Proposition \ref{FinitelyPresented}, all these groups are finitely presented. Moreover,  by {\cite[Theorem 4.35]{DidSMaSe99}}, the minimal number of relations of  $G_m(d)$ is ${m \choose 2}$.
\end{proof}
\begin{rem}  Note that one can give an explicit presentation for $G_m(d)$. For example, the uniform pro-$p$ group associated to the powerful $\mathbb{Z}_p$-Lie algebra $p^2L_4(d)$ is given by the presentation
\begin{displaymath}
G_4(d) = \langle y, z_2, z_3, z_4 ~ | ~ [z_2, z_3]=1, [z_3, z_4]=1,
\end{displaymath}
\begin{displaymath}
[z_4, z_2]=1, [z_2, y] = z_4^{dp^2}, [z_3, y] = z_3^{p^2}, [z_4, y] = z_2^{p^2}  \rangle.
\end{displaymath}
\end{rem}

\bigskip

\ackn  I am grateful to Slobodan Tanusevski for his valuable comments.

\bibliographystyle{plain}

\end{document}